\documentclass[10pt,twoside]{siamltex}
\usepackage{amsmath, amssymb,amsfonts}							
\usepackage{hyperref}
\usepackage{kbordermatrix}									
\usepackage{changepage}
\usepackage{enumitem}
\usepackage{multirow}

\allowdisplaybreaks

\newcommand{\bb}[1]{\mathbb{#1}}								
\newcommand{\mat}[2]{\textsl{M}_{#1}(#2)}						
\newcommand{\inv}[1]{#1^{-1}}								

\newcommand{\sr}[1]{\rho\left(#1\right)}							
\newcommand{\sig}[1]{\sigma \left( #1 \right)}						

\newcommand{\hyp}[2]{#1 \hyperref[#2]{\ref*{#2}}}					
\newcommand{\sulei}{Sule\u{\i}manova}

\newtheorem{example}[theorem]{Example}
\newtheorem{mydef}[theorem]{Definition}
\newtheorem{problem}[theorem]{Problem}

\newcommand{\reals}{\bb{R}}

\begin{document}



\title{Realizing Suleimanova Spectra via Permutative Matrices\thanks{}}

\author{
Pietro Paparella\thanks{Division of Engineering and Mathematics, University of Washington Bothell, Bothell, WA 98011-8246, USA (pietrop@uw.edu)}}

\pagestyle{myheadings}
\markboth{P.\ Paparella}{Realizing Suleimanova Spectra via Permutative Matrices}
\maketitle

\begin{abstract}
A \emph{permutative matrix} is a square matrix such that every row is a permutation of the first row. A constructive version of a result attributed to \sulei~is given via permutative matrices. In addition, we strengthen a well-known result by showing that all realizable spectra containing at most four elements can be realized by a permutative matrix or by a direct sum of permutative matrices. We conclude by posing a problem.
\end{abstract}



\section{Introduction} \label{sec:intro}

Introduced by \sulei~in \cite{s1949}, the longstanding \emph{real nonnegative inverse eigenvalue problem} (RNIEP) is to determine necessary and sufficient conditions on a set $\sigma = \{ \lambda_1, \dots, \lambda_n \} \subset \bb{R}$ so that $\sigma$ is the spectrum an $n$-by-$n$ entrywise nonnegative matrix. 

If $A$ is an $n$-by-$n$, nonnegative matrix with spectrum $\sigma$, then $\sigma$ said to be \emph{realizable} and the matrix $A$ is called a \emph{realizing matrix} for $\sigma$. It is well-known that if $\sigma$ is realizable, then  
\begin{align}
s_k (\sigma) &:= \sum_{i=1}^n \lambda_i^k \geq 0,~\forall~k \in \bb{N} \label{cond1} \\
\sr{\sigma} &:= \max_{1\leq i \leq n} |\lambda_i| \in \sigma. \label{cond2}
\end{align}
For additional background and results, see, e.g., \cite{eln2004,m1988} and references therein.

A set $\sigma = \{ \lambda_1, \dots, \lambda_n \} \subset \bb{R}$ is called a \emph{\sulei~spectrum} if $s_1(\sigma) \geq 0$ and $\sigma$ contains exactly one positive element. \sulei~\cite{s1949} announced (and loosely proved) that every such spectrum is realizable. Fiedler \cite{f1974} showed that every Sule\u{\i}ma\-nova spectrum is \emph{symmetrically realizable} (i.e., realizable by a symmetric nonnegative matrix), however, his proof is by induction and does not explicitly yield a realizing matrix for all orders. In \cite{jp_pre}, Johnson and Paparella provide a constructive version of Fiedler's result for Hadamard orders.

Friedland \cite{f1978} and Perfect \cite{p1953} proved \sulei's result via companion matrices (for other proofs, see references in \cite{f1978}). In particular, the coefficients $c_0,c_1,\dots,c_{n-1}$ of the polynomial $p(t) := \prod_{k=1}^n (t - \lambda_k) = t^n + \sum_{k=0}^{n-1} c_k t^k$ are nonpositive so that the companion matrix of $p$ is nonnegative. As noted in \cite[p.~1380]{rsr2001}, the construction of the companion matrix of $p$ requires evaluating the elementary symmetric functions at $(\lambda_1,\lambda_2,\dots,\lambda_n)$, a computation with $\mathcal{O}(2^n)$ complexity. 

The computation of a realizing matrix for a realizable spectrum is of obvious interest for numerical purposes, but for many known theoretical results, a realizing matrix is not readily available. Indeed, according to Chu: 
\begin{adjustwidth}{2.5em}{0pt}
Very few of these theoretical results are ready for implementation to actually compute [the realizing] matrix. The most constructive result we have seen is the sufficient condition studied by Soules \cite{s1983}. But the condition there is still limited because the construction depends on the specification of the Perron vector -- in particular, the components of the Perron eigenvector need to satisfy certain inequalities in order for the construction to work. \cite[p.~18]{c1998}. 
\end{adjustwidth}

In this work, we provide a constructive version of Sule\u{\i}manova's result via \emph{permutative matrices}. The paper is organized as follows: \hyp{Section}{sec:not} contains notation and definitions; \hyp{Section}{sec:main} contains the main results; in \hyp{Section}{sect:connect_rniep} we show that if $\sigma = \{ \lambda_1,\dots, \lambda_n\}$ satisfies \eqref{cond1} and \eqref{cond2}, then $\sigma$ is realizable by a permutative matrix or by a direct sum of permutative matrices; and we conclude by posing a problem in \hyp{Section}{sec:conc}.  

\section{Notation}
\label{sec:not}

The set of $m$-by-$n$ matrices with entries from a field $\bb{F}$ (in this paper, $\bb{F}$ is either $\bb{C}$ or $\bb{R}$) is denoted by $\mat{m,n}{\bb{F}}$ (when $m = n$, $\mat{n,n}{\bb{F}}$ is abbreviated to $\mat{n}{\bb{F}}$). For $A = [a_{ij}] \in \mat{n}{\bb{C}}$, $\sig{A}$ denotes the \emph{spectrum} of $A$. 

The set of $n$-by-$1$ column vectors is identified with the set of all $n$-tuples with entries in $\bb{F}$ and thus denoted by $\bb{F}^m$. Given $x \in \bb{F}^n$, $x_i$ denotes the $i\textsuperscript{th}$ entry of $x$.

For the following, the size of each object will be clear from the context in which it appears:
\begin{itemize}
\item $I$ denotes the identity matrix;  
\item $e$ denotes the all-ones vector; and
\item $J$ denotes the all-ones matrix, i.e., $J = e e^\top$. 
\end{itemize}

\begin{mydef}
\label{def:perm}
{\rm For $x \in \bb{C}^n$ and permutation matrices $P_2,\dots,P_n \in \mat{n}{\bb{R}}$, a \emph{permutative matrix}\footnote{Terminolgy due to Charles R.~Johnson.} is any matrix of the form
\begin{equation*}
\begin{bmatrix}
x^\top          \\
(P_2 x)^\top    \\
\vdots          \\
(P_n x)^\top    \\
\end{bmatrix} \in \mat{n}{\bb{C}}.
\end{equation*}}
\end{mydef}
According to \hyp{Definition}{def:perm}, all one-by-one matrices are considered permutative.

\section{Main Results} 
\label{sec:main}

We begin with the following lemmas.

\begin{lemma} 
\label{lem:perm}
For $x \in \mathbb{C}^n$, let 
\begin{equation*}
P = P_x = 
\kbordermatrix{                         
  & 1 & 2 & \cdots & i & \cdots & n       \\
1 & x_1 & x_2 & \cdots & x_i & \cdots & x_n \\
2 & x_2 & x_1 & \cdots & x_i & \cdots & x_n \\
\vdots & \vdots & \vdots & \ddots & \vdots & & \vdots \\
i & x_i & x_2 & \cdots & x_1 & \cdots & x_n \\
\vdots & \vdots & \vdots &  & \vdots & \ddots & \vdots\\
n & x_n & x_2 & \cdots &x_i & \cdots & x_1
}
=
\begin{bmatrix}
x^\top                  \\
(P_{\alpha_2} x)^\top   \\
\vdots                  \\
(P_{\alpha_i} x)^\top   \\
\vdots                  \\
(P_{\alpha_n} x)^\top    \\
\end{bmatrix},
\end{equation*}
where $P_{\alpha_i}$ is the permutation matrix corresponding to the permutation $\alpha_i$ defined by $\alpha_i(x) = (1i)$, $i= 2,\dots,n$. Then $\sigma(P) = \{ s, \delta_2, \dots, \delta_n \}$,
where $s := \sum_{i=1}^n x_i$ and $\delta_i := x_1 - x_i$, $i=2,\dots,n$.
\end{lemma}

\begin{proof}
Since every row sum of $P$ is $s$, it follows that $Pe = se$, i.e., $s \in \sigma(P)$.

Since
\begin{equation*}
P - \delta_i I = 
\kbordermatrix{                         
  & 1 & 2 & \cdots & i & \cdots & n       \\
1 & x_i & x_2 & \cdots & x_i & \cdots & x_n \\
2 & x_2 & x_i & \cdots & x_i & \cdots & x_n \\
\vdots & \vdots & \vdots & \ddots & \vdots & & \vdots \\
i & x_i & x_2 & \cdots & x_i & \cdots & x_n \\
\vdots & \vdots & \vdots & & \vdots & \ddots & \vdots\\
n & x_n & x_2 & \cdots &x_i & \cdots & x_i
},
\end{equation*}
it follows that the homogeneous linear system $(P - \delta_i I) \hat{x} = 0$ has a nontrivial solution (notice that the first and $i\textsuperscript{th}$ rows of $P- \delta_iI$ are identical). Thus, $\delta_i \in \sigma(P)$.

Moreover, if 
\begin{equation*}
v_i := 
\kbordermatrix{
                \\
1 & x_i         \\
\vdots & \vdots \\
i-1 & x_i       \\
i & x_1 - s     \\
i+1 & x_i       \\
\vdots & \vdots \\
n & x_i},~i=2,\dots,n
\end{equation*}
then 
\begin{equation*}
Pv_i = 
\kbordermatrix{                 \\
1 & x_i (s - x_i) + x_i (x_1 -s)    \\
\vdots & \vdots                          \\
i-1 & x_i (s - x_i) + x_i (x_1 -s)    \\
i & x_i (s-x_1) + x_1(x_1 - s)        \\
i+1 & x_i (s - x_i) + x_i (x_1 -s)    \\
\vdots & \vdots                          \\
n & x_i (s - x_i) + x_i (x_1 -s)}
=
(x_1 - x_i)
\begin{bmatrix}
x_i     \\
\vdots  \\
x_i     \\
x_1 - s \\
x_i     \\
\vdots  \\
x_i     \\
\end{bmatrix}
= \delta_i v_i,
\end{equation*}
so that $(\delta_i,v_i)$ is a right-eigenpair for $P$.
\end{proof}

\begin{lemma} 
\label{lem:papmat}
If
\begin{equation*}
M = M_n :=
\begin{bmatrix}
1 & e^\top \\
e & -I 
\end{bmatrix} \in \mat{n}{\reals},~n\geq2,
\end{equation*}
then
\begin{equation*}
\inv{M}= \inv{M_n} =
\frac{1}{n}
\begin{bmatrix}
1 & e^\top \\
e & J-nI
\end{bmatrix}.
\end{equation*}
\end{lemma}

\begin{proof}
Clearly,
\begin{equation*}
n M\inv{M} 
=
\begin{bmatrix}
1 & e^\top \\
e & -I 
\end{bmatrix}
\cdot
\begin{bmatrix}
1 & e^\top \\
e & J-nI
\end{bmatrix}           
=
\begin{bmatrix}
n & e^\top + e^\top(J-nI) \\
0 & nI
\end{bmatrix},             
\end{equation*}
but $e^\top + e^\top(J-nI) = e^\top + (n-1) e^\top- ne^\top = 0$;
dividing through by $n$ establishes the result.
\end{proof}

\begin{theorem}[Sule\u{\i}manova]
\label{thm:sulei}
Every Sule\u{\i}manova spectrum is realizable. 
\end{theorem}

\begin{proof}
Let $\sigma = \{ \lambda_1,\dots, \lambda_n \}$ be a Sule\u{\i}manova spectrum and assume, without loss of generality, that $\lambda_1 \geq 0 \geq \lambda_2 \geq \cdots \geq \lambda_n$. If $\lambda := [\lambda_1~\lambda_2~\cdots~\lambda_n]^\top \in \reals^n$, then, following \hyp{Lemma}{lem:papmat}, the solution $x$ of the linear system
\begin{equation*}
\left\{
\begin{array}{*{9}{c}}
 x_1 & + & x_2 & + & \cdots & + & x_n & = & \lambda_1       \\
 x_1 & - & x_2 & & & &               & = & \lambda_2   \\
    & & & & & & & \vdots                            \\
 x_1 & & & & & - & x_n & = & \lambda_n
\end{array} \right.
\end{equation*}
is given by
\begin{equation*}
x
=\inv{M}\lambda  
=\frac{1}{n}
\begin{bmatrix}
s_1(\sigma)                \\
s_1(\sigma) - n \lambda_2  \\
\vdots                      \\
s_1(\sigma) - n\lambda_n     
\end{bmatrix}.
\end{equation*}
which is clearly nonnegative. Following \hyp{Lemma}{lem:perm}, the nonnegative matrix $P_x$ realizes $\sigma$.
\end{proof}

\begin{example}
{\rm If $\sigma = \{10,-1,-2,-3\}$, then $\sigma$ is realizable by  
\begin{equation*}
\begin{bmatrix}
1 & 2 & 3 & 4   \\
2 & 1 & 3 & 4   \\
3 & 2 & 1 & 4   \\
4 & 2 & 3 & 1
\end{bmatrix}.
\end{equation*}}
\end{example}

\begin{corollary}
If $\sigma = \{ \lambda_1,-\lambda_2,\dots, -\lambda_n \}$ is a Sule\u{\i}manova spectrum such that $s_1(\sigma) = 0$ and $\lambda_1 > 0$, then the $n$-by-$n$ nonnegative matrix
\begin{equation*}
P := 
\begin{bmatrix}                         
0 & \lambda_2 & \cdots & \lambda_i & \cdots & \lambda_n \\
\lambda_2 & 0 & \cdots & \lambda_i & \cdots & \lambda_n \\
\vdots & \vdots & \ddots & \vdots & & \vdots \\
\lambda_i & \lambda_2 & \cdots & 0 & \cdots & \lambda_n \\
\vdots & \vdots &  & \vdots & \ddots & \vdots\\
\lambda_n & \lambda_2 & \cdots &\lambda_i & \cdots & 0
\end{bmatrix}
\end{equation*}
realizes $\sigma$.
\end{corollary}

\begin{example}
{\rm If $\sigma = \{6,-1,-2,-3\}$, then $\sigma$ is realizable by  
\begin{equation*}
\begin{bmatrix}
0 & 1 & 2 & 3   \\
1 & 0 & 2 & 3   \\
2 & 1 & 0 & 3   \\
3 & 1 & 2 & 0
\end{bmatrix}.
\end{equation*}}
\end{example}

\section{Connection to the RNIEP}
\label{sect:connect_rniep}


It is well-known that for $1 \leq n \leq 4$, conditions \eqref{cond1} and \eqref{cond2} are also sufficient for realizability (see, e.g., \cite{jp_pre, ll1978-79}). In this section, we strengthen this result by demonstrating that the realizing matrix can be taken to be permutative or a direct sum of permutative matrices.

\begin{theorem}
If $\sigma = \{ \lambda_1, \dots, \lambda_n \} \subset \bb{R}$ and $1 \leq n \leq 4$, then $\sigma$ is realizable if and only if $\sigma$ satisfies \eqref{cond1} and \eqref{cond2}. Futhermore, the realizing matrix can be taken to be permutative or a direct sum of permutative matrices.
\end{theorem}

\begin{proof}
Without loss of generality, assume that $\sr{\sigma} = 1$.

The case when $n=1$ is trivial, but it is worth mentioning that $\sigma = \{1\}$ is realized by the permutative matrix $[1]$.

If $\sigma = \{ 1, \lambda \}$, $-1 \leq \lambda \leq 1$, then the permutative matrix
\begin{equation*}
\frac{1}{2}
\begin{bmatrix}
1 + \lambda & 1 - \lambda \\
1 - \lambda & 1 + \lambda
\end{bmatrix}
\end{equation*} 
realizes $\sigma$.

As established in \cite{jp_pre}, if $\sigma = \{1, \mu, \lambda \}$, where $-1 \leq \mu, \lambda \leq 1$, then the matrix 
\begin{equation*}
\begin{bmatrix}
(1 + \lambda)/2 & (1 - \lambda)/2 & 0 \\
(1 - \lambda)/2 & (1 + \lambda)/2 & 0 \\
0 & 0 & \mu
\end{bmatrix}
\end{equation*}
realizes $\sigma$ when $1 \geq \mu \geq \lambda \geq 0$ or $1 \geq \mu \geq 0 > \lambda$. Notice that this matrix is a direct sum of permutative matrices. If $0 > \mu \geq \lambda$, then, following \hyp{Theorem}{thm:sulei}, $\sigma$ is realizable by a permutative matrix. 

When $n=4$, all realizable spectra can be realized by matrices of the form
\begin{equation*}
\begin{bmatrix}
a + b & a - b & 0 & 0 \\
a - b & a + b  & 0 & 0 \\
0 & 0 & c + d & c - d \\
0 & 0 & c - d & c + d
\end{bmatrix} \mbox{ or }
\begin{bmatrix}
a & b & c & d \\
b & a & d & c \\
c & d & a & b \\
d & c & b & a
\end{bmatrix}
\end{equation*}
(for full details, see \cite[pp.~10--11]{jp_pre}). 
\end{proof}

\section{Concluding Remarks}
\label{sec:conc}

In \cite{f1978}, Fiedler posed the \emph{symmetric nonnegative inverse eigenvalue problem} (SNIEP), which requires the realizing matrix to be symmetric. Obviously, if $\sigma =\{ \lambda_1, \dots, \lambda_n \}$ is a solution to the SNIEP, then it is a solution to the RNIEP. In \cite{jll1996}, Johnson, Laffey, and Loewy that showed that the RNIEP strictly contains the SNIEP when $n\geq5$. It is in the spirit of this problem that we pose the following. 

\begin{problem}
\label{prob:conc}
Can all realizable real spectra be realized by a permutative matrix or by a direct sum of permutative matrices?
\end{problem}

At this point there is no evidence that suggests an affirmative answer to \hyp{Problem}{prob:conc}; however, a negative answer could be just as difficult. One possibility to establish a negative answer, communicated to me by R.~Loewy, is to find an \emph{extreme nonnegative matrix} \cite{l1998} with a real spectrum that can not be realized by a permutative matrix, or a direct sum of permutative matrices.  



\section*{Acknowledgment} I wish to thank the anonymous referee and Raphael Loewy for their comments on improving the first-draft.  

\bibliographystyle{siam}

\begin{thebibliography}{10}

\bibitem{c1998}
M.~T.~Chu.
\newblock Inverse eigenvalue problems.
\newblock {\em SIAM Rev.}, 40(1):1--39, 1998.

\bibitem{eln2004}
P.~D.~Egleston, T.~D.~Lenker, and S.~K.~Narayan.
\newblock The nonnegative inverse eigenvalue problem.
\newblock {\em Linear Algebra Appl.}, 379:475--490, 2004.
\newblock Tenth Conference of the International Linear Algebra Society.

\bibitem{f1974}
M.~Fiedler.
\newblock Eigenvalues of nonnegative symmetric matrices.
\newblock {\em Linear Algebra and Appl.}, 9:119--142, 1974.

\bibitem{f1978}
S.~Friedland.
\newblock On an inverse problem for nonnegative and eventually nonnegative
  matrices.
\newblock {\em Israel J. Math.}, 29(1):43--60, 1978.

\bibitem{jll1996}
C.~R. Johnson, T.~J. Laffey, and R.~Loewy.
\newblock The real and the symmetric nonnegative inverse eigenvalue problems
  are different.
\newblock {\em Proc. Amer. Math. Soc.}, 124(12):3647--3651, 1996.

\bibitem{jp_pre}
C.~R.~Johnson and P.~Paparella.
\newblock Perron spectratopes and the real nonnegative inverse eigenvalue problem.
\newblock To appear in Linear Algebra Appl., \href{http://arxiv.org/abs/1508.07400}{\texttt{arXiv:1508.07400}}.

\bibitem{ll1978-79}
R.~Loewy and D.~London.
\newblock A note on an inverse problem for nonnegative matrices.
\newblock {\em Linear and Multilinear Algebra}, 6(1):83--90, 1978/79.

\bibitem{l1998}
T.~J.~Laffey.
\newblock Extreme nonnegative matrices.
\newblock {\em Linear Algebra and Appl.}, 275/276:349--357, 1998.

\bibitem{m1988}
H.~Minc.
\newblock {\em Nonnegative matrices}.
\newblock Wiley-Interscience Series in Discrete Mathematics and Optimization.
  John Wiley \& Sons, Inc., New York, 1988.
\newblock A Wiley-Interscience Publication.

\bibitem{p1953}
H.~Perfect.
\newblock Methods of constructing certain stochastic matrices.
\newblock {\em Duke Math. J.}, 20:395--404, 1953.

\bibitem{rsr2001}
O.~Rojo, R.~Soto, and H.~Rojo
\newblock Fast construction of a symmetric nonnegative matrix with a prescribed spectrum.
\newblock {\em Comput. Math. Appl.}, 42(10--11):1379--1391, 2001.

\bibitem{s1983}
G.~W.~Soules.
\newblock Constructing symmetric nonnegative matrices.
\newblock {\em Linear and Multilinear Algebra}, 13(3):241--251, 1983.

\bibitem{s1949}
H.~R.~Sule{\u\i}manova.
\newblock Stochastic matrices with real characteristic numbers.
\newblock {\em Doklady Akad. Nauk SSSR (N.S.)}, 66:343--345, 1949.

\end{thebibliography}

\end{document}